\documentclass[12pt, reqno]{amsart}
\usepackage{amsmath, amsthm, amscd, amsfonts, amssymb, graphicx, color}
\usepackage{mathrsfs}
\usepackage[bookmarksnumbered, colorlinks, plainpages]{hyperref}
\usepackage{tikz}
\usepackage{subcaption}
\usetikzlibrary{decorations.pathreplacing,angles,quotes}

\definecolor{cof}{RGB}{219,144,71}
\definecolor{pur}{RGB}{186,146,162}
\definecolor{greeo}{RGB}{91,173,69}
\definecolor{greet}{RGB}{52,111,72}
\usepackage{xcolor}
\hypersetup{colorlinks=true,linkcolor=red, anchorcolor=green, citecolor=cyan, urlcolor=red, filecolor=magenta, pdftoolbar=true}

\newtheorem{theorem}{Theorem}[section]
\newtheorem{lemma}[theorem]{Lemma}
\newtheorem{proposition}[theorem]{Proposition}
\newtheorem{cor}[theorem]{Corollary}
\newtheorem{definition}[theorem]{Definition}

\theoremstyle{remark}
\newtheorem{remark}[theorem]{\bf{Remark}}
\numberwithin{equation}{section}
\allowdisplaybreaks
\begin{document}

\title [Berezin sectorial operators and Berezin number inequalties]{An introduction of Berezin sectorial operators and its application to Berezin number inequalities}

\author[S. Mahapatra, S. Mukherjee, A.Sen, R. Birbonshi and K. Paul]{Saikat Mahapatra, Sweta Mukherjee, Anirban Sen, Riddhick Birbonshi and Kallol Paul}
	
\address[Mahapatra]{Department of Mathematics, Jadavpur University, Kolkata 700032, West Bengal, India}
\email{smpatra.lal2@gmail.com}

\address[Mukherjee]{Department of Mathematics, Jadavpur University, Kolkata 700032, West Bengal, India}
\email{sweta.bankati@gmail.com}

\address[Sen] {Mathematical Institute, Silesian University in Opava, Na Rybn\'{\i}\v{c}ku 1, 74601 Opava, Czech Republic}
\email{anirbansenfulia@gmail.com; Anirban.Sen@math.slu.cz}

\address[Birbonshi] {Department of Mathematics, Jadavpur University, Kolkata 700032, West Bengal, India}
\email{riddhick.math@gmail.com}

\address[Paul] {Vice-Chancellor, Kalyani University, West Bengal 741235 and 
		Professor (on lien), Department of mathematics, Jadavpur University, Kolkata 700032, West Bengal, India}
	\email{kalloldada@gmail.com}




\maketitle	

\begin{abstract}  
  We introduce a new class of operators, called Berezin sectorial operators, which generalizes classical sectorial operators. We provide examples on the Hardy-Hilbert space showing that there exist operators that are Berezin sectorial but not sectorial and that the Berezin sectorial index can be strictly smaller than the classical one. We derive Berezin number inequalities for this class, including a weak version of the power inequality, and study geometric properties of the Berezin range for finite-rank and weighted shift operators on the Dirichlet space. We also raise the question of whether similar constructions are possible for composition-differentiation operators on the Dirichlet space.
\end{abstract}
\tableofcontents
\section{Introduction}
An operator is said to be sectorial if its numerical range lies within a cone in the right half of the complex plane with vertex at the origin. Sectorial operators have attracted significant attention due to their important role in the study of numerical radius inequalities. Over the years, many researchers have investigated these operators from various geometric perspectives and have developed several interesting bounds for the numerical radius that significantly refine existing results, see \cite{On sectorial matrices,Numerical radius,Drury,Norm of an,A geometric}.
The primary objective of this paper is to introduce Berezin sectorial operators, which includes sectorial operators as a special case. We provide illustrative examples of these operators in the Hardy-Hilbert space and establish Berezin number inequalities for this class which refines existing inequalities for general bounded linear operators on reproducing kernel Hilbert spaces.

Let $\mathscr{B(H)}$ denote the $C^*$-algebra of all bounded linear operators on a complex separable Hilbert space $\mathscr{H}$ with the usual inner product $\langle\cdot,\cdot\rangle$. Any $T\in \mathscr{B(H)}$, can be decomposed uniquely as $T=\Re(T)+i\Im(T)$, where $\Re(T)=\frac{1}{2}(T+T^*)$ and $\Im(T)=\frac{1}{2i}(T-T^*)$. This decomposition is known as the Cartesian decomposition. $T\in \mathscr{B(H)}$ is said to be positive operator if $\langle Tx,x\rangle > 0$ for all $x (\neq 0)\in \mathscr{H}$ and it is denoted by $T>0$. For $T\in \mathscr{B(H)}$, $|T|$ stands for the positive operator $(T^*T)^{\frac{1}{2}} $. The numerical range and the numerical radius of $T\in \mathscr{B(H)}$ are, respectively, denoted by $W(T)$ and $w(T)$ and are respectively defined as $W(T)=\{\langle Tx,x\rangle: x\in \mathscr{H}, \,\,\|x\|=1\}$ and $w(T)=\sup\{|\langle Tx,x\rangle|:x\in \mathscr{H}, \,\,\|x\|=1\}$.

 A reproducing kernel Hilbert space $ \mathscr{H}$ on a non-empty set $\mathscr{X}$ is a Hilbert space of all complex valued functions on $\mathscr{X}$ with the property that for every $\lambda \in \mathscr{X}$, the linear evaluation functional on $\mathscr{H}$ given by $\phi \to \phi(\lambda),$ is continuous (see \cite{PR_Book_2016}). The Riesz representation theorem ensures that for each $\lambda \in \mathscr{X}$, there exists a unique $ k_\lambda \in \mathscr{H}$ such that $\phi(\lambda)=\langle \phi, k_\lambda \rangle,$ for all $\phi \in  \mathscr{H}$. The collection of functions $\{k_\lambda :  \lambda \in \mathscr{X} \}$ is the set of all reproducing kernels of $ \mathscr{H}.$ The set of all normalized reproducing kernels of $ \mathscr{H}$ is the collection of functions $\{\hat{k}_{\lambda}=k_\lambda/\|k_\lambda\| :  \lambda\in \mathscr{X}\}.$

In this article, $\mathscr{H}$ stands for a reproducing kernel Hilbert space on the set $\mathscr{X}.$ For every $T \in  \mathscr{B}( \mathscr{H}),$ the Berezin transform of $T$, (see \cite{covariant,BER1}) is the function $\widetilde{T}$ on $\mathscr{X}$ defined by
$$\widetilde{T}(\lambda)=\langle T\hat{k}_{\lambda},\hat{k}_{\lambda} \rangle~~\text{for all $\lambda \in \mathscr{X}$}.$$
The Berezin range of $T$ is defined by
$$\textbf{Ber}(T):= \{\widetilde{T}(\lambda) : \lambda \in \mathscr{X}\},$$
see \cite{Reproducing}.
For $T \in  \mathscr{B}( \mathscr{H}),$ the quantities $\textbf{ber}(T)$ and $\|T\|_{ber},$ referred to as the Berezin number and the Berezin norm of 
$T$, respectively, are given by  
$$\textbf{ber}(T):=\sup \left\{|\widetilde{T}(\lambda)| : \lambda \in \mathscr{X} \right\}$$
and 
$$\|T\|_{ber}:=\sup\left\{|\langle T\hat{k}_{\lambda},\hat{\mu}_{y}\rangle | : \lambda,\mu \in \mathscr{X}\right\},$$
see \cite{BY_JIA_2020,Berezin symbol}. It is clear from the definition that for any $T\in \mathscr{B(H)}$ ,\[\textbf{Ber}(T)\subseteq W(T)\,\,\, \mbox{and}\,\,\, \textbf{ber}(T)\le w(T). \]
Unlike the numerical range, the Berezin range is not, in general, convex. The convexity of the Berezin range was first studied in \cite{convexity} and subsequently explored in several works \cite{Bulletin des,Composition,Reproducing,berezin toeplitz}. 
Moreover, the equality of the Berezin number and Berezin for positive operators is proved in \cite{pintuani}.
For recent developments in Berezin number inequalities, we refer to \cite{pintuRacsam,Barik,mahapatra,Majee}.

This article is organized as follows. Together with the introduction, it consists of five sections. In Section \ref{S3}, we introduce a new class of operators, called Berezin sectorial operators, which generalizes the classical notion of sectorial operators. This section also presents several examples in the Hardy–Hilbert space and explains the motivation for defining this new class. In Section \ref{S4}, we derive several Berezin number inequalities for Berezin sectorial operators, refining a number of known inequalities. Section \ref{S5} is devoted to power inequalities for the Berezin number, where we establish new inequalities for Berezin sectorial operators whose real and imaginary parts satisfy corresponding power inequalities, and also develop a weak version of these inequalities. Lastly, Section \ref{S6} outlines some fundamental characteristics of the Berezin range, such as convexity and symmetry about the real and imaginary axes, for finite-rank operators and weighted shift operators on the Dirichlet space. Since the Berezin range of these operators exhibits properties similar to those observed in the Hardy–Hilbert space, we conclude by posing the question of whether analogous constructions of Berezin sectorial operators can be achieved for certain composition–differentiation operators on the Dirichlet space.

\section{Berezin sectorial operators} \label{S3}

In this section, we begin by recalling the well-known definition of a sectorial operator as given in \cite{def}. 

\begin{definition}
   Let  $\theta\in[0,\frac{\pi}{2})$ and $S_\theta=\{z\in\mathbb{C}: \,\, |\arg z|\le \theta\}$
 be a sector in the complex plane $\mathbb{C}$, with the vertex at the origin and the semi-angle $\theta$. Then a linear operator $T$ in a
 Hilbert space $\mathscr{H}$ is called sectorial with vertex at $z = 0$ and the semi-angle $\theta$, if $W(T)\subseteq S_\theta$.
\end{definition}

Motivated by the definition of the sectorial operator, we define a new class of bounded linear operators, termed Berezin sectorial operators as:

\begin{definition}
   A linear operator $T$ in a
 reproducing kernel Hilbert space $\mathscr{H}$ is called Berezin sectorial with vertex at $z = 0$ and the semi-angle $\theta\in[0,\frac{\pi}{2})$, if $\textbf{Ber}(T)\subseteq S_\theta$, i.e., the  Berezin range of  $T$  lies entirely in a cone in the
 right half complex plane, with vertex at the origin and half
angle $\theta$. This class of operators will be denoted by $\Pi^{\textbf{Ber}}_{\theta}.$
\end{definition}
It follows directly from the definitions that every sectorial operator is a Berezin sectorial operator, making the latter a natural generalization of the former. However, the converse does not hold, not every Berezin sectorial operator is sectorial. To illustrate Berezin sectorial operators, we first consider composition-differentiation operators on Hardy-Hilbert spaces, for which it is necessary to estimate both the Berezin range and the numerical range.

Recall that, the Hardy-Hilbert space on the open unit disk $\mathbb{D}$ is denoted by $H^2(\mathbb D),$ and is defined as 
$$H^2(\mathbb D)=\left\{f : f(z)=\sum_{n=0}^{\infty}a_nz^n~~\text {with}~~ \sum_{n=0}^{\infty}|a_n|^2 < \infty\right\}.$$ 
$H^2(\mathbb D)$ is a reproducing kernel Hilbert space, where the kernel function (called Szeg\H{o} kernel) at $w \in \mathbb D$ is given by
$$k_w(z)=\frac{1}{1-\bar wz}~~\text{for all $z \in \mathbb D$}.$$
Furthermore, the kernel function for the point evaluations of the first derivative is given by
  $$k'_w(z)=\frac{d}{d\bar w}(\frac{1}{1-\bar wz})=\frac{z}{1-\bar wz}~~\text{for all $z \in \mathbb D$},$$ see \cite[Th. 2.16]{Cowen}.
For an analytic self-map $\phi$ on $\mathbb{D}$, the composition-differentiation operator $D_{\phi}$ on $H^2(\mathbb D)$ is defined by $D_{\phi}(f)=f'\circ\phi.$ 
It is well-known that if $\|\phi\|_{\infty}<1$ then $D_{\phi}$ is bounded on $H^2$ (see \cite{Ohno}) and for $\phi(z)=\rho z$ with $0<\rho<1$,
\[\|D_{\phi}\|=\Big\lfloor\frac{1}{1-\rho}\Big\rfloor \rho^{\big\lfloor\frac{1}{1-\rho}\big\rfloor-1},\] 
where $\lfloor\cdot\rfloor$ denotes the greatest integer function (See \cite[Th. 2]{Fatehi}).
Now, we determine the Berezin range of the operator $D_{\phi}$ acting on $H^2(\mathbb D),$ where  $\phi(z)=\rho z$ with $\rho\in(0,1).$

\begin{proposition}
    If $\phi(z)=\rho z$ with $\rho\in(0,1)$ then $$\textbf{Ber}(D_{\phi})=\overline{B(0,r_1(\rho))},$$ where 
    $$r_1(\rho)=\frac{(3-\sqrt{9\rho^2-14\rho+9}-\rho)\sqrt{6\rho+2\sqrt{9\rho^2-14\rho+9}-6}}{\sqrt{\rho}(3\rho+\sqrt{9\rho^2-14\rho+9}-5)^2}.$$
\end{proposition}
 \begin{proof}
     Berezin symbol of $D_{\phi}$ at $z=re^{i\zeta} \in \mathbb{D}$ is given by
     \begin{align*}
         \widetilde{D_{\phi}}(z)=\frac{ 1}{\|k_z\|^2}\langle D_{\phi}{k_z},{k_z}\rangle=(1-|z|^2)k'_z(\rho z)=
         (1-r^2)\frac{\rho r}{(1-\rho r^2)^2}e^{i\zeta}.
     \end{align*}
     Consider the function $f:[0,1]\rightarrow \mathbb{R}^+$ given by $f(r)=(1-r^2)\frac{\rho r}{(1-\rho r^2)^2}.$ Clearly, $f$ is continuous and becomes zero at both of the endpoints. So, $f$ attains its maximum value in $(0,1)$. 
     The maximum value of $f$ is given by $$\max_{r \in(0,1)}f(r)=\frac{(3-\sqrt{9\rho^2-14\rho+9}-\rho)\sqrt{6\rho+2\sqrt{9\rho^2-14\rho+9}-6}}{\sqrt{\rho}(3\rho+\sqrt{9\rho^2-14\rho+9}-5)^2}=r_1(\rho).$$\\
     The continuity of $f$ implies that $\textbf{Ber}(D_{\phi})=\overline{B(0,r_1(\rho))}.$
     
 \end{proof}  
 Next we find the numerical range of $D_{\phi}$, where  $\phi(z)=\rho z$ with $\rho\in(0,1).$
 For this we first note that the matrix representation of $D_\phi$ corresponding to the orthonormal basis $\{e_n\}_{n=0}^{\infty}$ where $e_n(z)=z^n,$ is given by
    \begin{align*}
      D_\phi= \begin{pmatrix}
           0&0&0&.&.&.&.\\
           1&0&0&.&.&.&.\\
           0&2\rho&0&.&.&.&.\\
           0&0&3\rho^2&.&.&.&.\\
           0&0&0&4\rho^3&.&.&.\\
           .&.&.&.&.&.&.\\
           .&.&.&.&.&.&.\\
           .&.&.&.&.&.&.
       \end{pmatrix}.
       \end{align*}
       Clearly, here $D_\phi$ is a bounded weighted shift operator. The numerical range of this weighted shift operator $D_{\phi}$ is unitarily equivalent to $e^{i\theta}D_\phi$ for all $\theta \in \mathbb R,$ hence numerical range of $D_{\phi}$ is a circular disk with centre at the origin (see \cite{Stout}). Also, we have the sequence of weights of this $D_\phi$ is bounded and converges to 0. So, $D_{\phi}$ is compact, and also $0\in W(D_{\phi})$. Therefore, the numerical range of $D_{\phi}$ is a closed circular disc with centre at the origin (\cite[Cor 1]{Lancaster}). 
       Thus, this observation characterizes the shape of the numerical range of $D_{\phi},$ but it does not provide any information about its size. Since determining the exact numerical radius is highly nontrivial, we instead establish lower and upper bounds for it.
       
       It is well-known that $\frac{\|T\|}{2} \leq w(T)$ (see \cite{Gustafson}). Consequently, this yields the following lower bound of $w(D_{\phi}),$ namely,
       \begin{eqnarray}\label{iifff}
             \frac{1}{2}\left\lfloor\frac{1}{1-\rho}\right\rfloor\rho^{\left\lfloor\frac{1}{1-\rho}\right\rfloor-1} \leq w(D_{\phi}).
       \end{eqnarray} 

       To obtain an upper bound, we recall a known estimate involving the Aluthge transformation, which will be used in our next result. Let $T$ be a bounded linear operator on a Hilbert space, and let $T=U|T|$ be the polar decomposition, where  $|T|=(T^*T)^\frac{1}{2}$ and $U$ is partial isometry. The Aluthge transformation $\widetilde{T}$ of $T$ is defined by $$\widetilde{T}=|T|^\frac{1}{2}U|T|^\frac{1}{2}.$$ It is proved in \cite[Th. 2.1]{Yamazaki} that the upper bound for the numerical radius of a bounded linear operator can be obtained via the numerical radius of its Aluthge transformation. In particular, the following inequality holds:
       \begin{eqnarray}\label{iiff}
             w(T)\le \frac12 (\|T\|+w(\widetilde{T})).
       \end{eqnarray}
       
  Now, we are in a position to prove the following inclusion.
 \begin{proposition}
         If $\phi(z)=\rho z$ with $\rho\in(0,1)$ then $$\overline{B(0,r_2(\rho))}\subseteq W(D_\phi)\subseteq\overline{B(0,r_3(\rho))},$$ where 
         $$r_2(\rho)= \frac{1}{2}\left\lfloor\frac{1}{1-\rho}\right\rfloor\rho^{\left\lfloor\frac{1}{1-\rho}\right\rfloor-1}$$ and $$r_3(\rho)=  \frac{1}{2}\left(\left\lfloor\frac{1}{1-\rho}\right\rfloor\rho^{\left\lfloor\frac{1}{1-\rho}\right\rfloor-1}+\sqrt{\left\lfloor\frac{1+\rho^2}{1-\rho^2}\right\rfloor}\rho^{\left\lfloor\frac{1+\rho^2}{1-\rho^2}\right\rfloor-\frac{1}{2}}\right).$$ 
 \end{proposition}

\begin{proof}
 The Aluthge transformation of $D_\phi$ is 
 $\widetilde{D_\phi}=|D_\phi|^\frac{1}{2}U|D_\phi|^\frac{1}{2},$ where
           %
\begin{align*}
     U=\begin{pmatrix}
            0&0&0&.&.&.&.\\
            1&0&0&.&.&.&.\\
            0&1&0&.&.&.&.\\
            .&.&.&.&.&.\\
            .&.&.&.&.&.\\
        \end{pmatrix}
\end{align*}
    and \begin{align*}
        |D_\phi|=(D_\phi^*D_\phi)^\frac{1}{2}=\begin{pmatrix}
            {w_1}&0&0&.&.&.\\
            0&{w_2}&0&.&.&.\\
            0&0&{w_3}&.&.&.\\
            .&.&.&.&.&.\\
            .&.&.&.&.&.\\
            .&.&.&.&.&.\\
        \end{pmatrix},      
        \end{align*} with $w_n=n\rho^{n-1}$, $n=1,2,\cdots$.
Therefore, 
\begin{align*}
    \widetilde{D_\phi}&=\begin{pmatrix}
            \sqrt{w_1}&0&0&.&.&.\\
            0&\sqrt{w_2}&0&.&.&.\\
            0&0&\sqrt{w_3}&.&.&.\\
            .&.&.&.&.&.\\
            .&.&.&.&.&.\\
            .&.&.&.&.&.\\
    \end{pmatrix}
    \begin{pmatrix}
        0&0&0&.&.&.&.\\
            1&0&0&.&.&.&.\\
            0&1&0&.&.&.&.\\
            .&.&.&.&.&.\\
            .&.&.&.&.&.\\
    \end{pmatrix}
    \begin{pmatrix}
            \sqrt{w_1}&0&0&.&.&.\\
            0&\sqrt{w_2}&0&.&.&.\\
            0&0&\sqrt{w_3}&.&.&.\\
            .&.&.&.&.&.\\
            .&.&.&.&.&.\\
            .&.&.&.&.&.\\
    \end{pmatrix}\\&=
    \begin{pmatrix}
        \sqrt{w_1w_2}&0&0&.&.&.\\
            0&\sqrt{w_2w_3}&0&.&.&.\\
            0&0&\sqrt{w_3w_4}&.&.&.\\
            .&.&.&.&.&.\\
            .&.&.&.&.&.\\
            .&.&.&.&.&.\\
    \end{pmatrix}\\&=
    \begin{pmatrix}
        {\alpha_1}&0&0&.&.&.\\
            0&{\alpha_2}&0&.&.&.\\
            0&0&{\alpha_3}&.&.&.\\
            .&.&.&.&.&.\\
            .&.&.&.&.&.\\
            .&.&.&.&.&.\\
    \end{pmatrix}, \,\,\alpha_n=\sqrt{w_nw_{n+1}}=\sqrt{n(n+1)}\rho^{n-\frac{1}{2}}.
\end{align*}
Since, $\|\widetilde{D_\phi}\|=\sup\{\alpha_n:n\in\mathbb{N}\}$ so now we calculate the supremum of the sequence $\{\alpha_n\}$.
Note that the function $f:(1,\infty)\rightarrow \mathbb{R}^+$ given by  $f(x)=x(x+1)\rho^{2x-1}$ has exactly one critical point in $(1, \infty)$ and the point is a local maximum. The maximum will be attained at the greatest natural number $n\geq 2$ such that
\[n(n-1)\rho^{2n-3} \leq n(n+1)\rho^{2n-1},\]
equivalently,
\[n \leq \left\lfloor\frac{1+\rho^2}{1-\rho^2}\right\rfloor.\]
Thus, we have
$$\|\widetilde{D_\phi}\|=\sqrt{\left\lfloor\frac{1+\rho^2}{1-\rho^2}\right\rfloor\left(\left\lfloor\frac{1+\rho^2}{1-\rho^2}\right\rfloor+1\right)}\rho^{\left\lfloor\frac{1+\rho^2}{1-\rho^2}\right\rfloor-\frac{1}{2}}.$$
Now, from the inequality \eqref{iiff}, we obtain
\begin{eqnarray}\label{iifffff}
 w(D_\phi)\le\frac{1}{2}\left(\left\lfloor\frac{1}{1-\rho}\right\rfloor\rho^{\left\lfloor\frac{1}{1-\rho}\right\rfloor-1}+\sqrt{\left\lfloor\frac{1+\rho^2}{1-\rho^2}\right\rfloor}\rho^{\lfloor\frac{1+\rho^2}{1-\rho^2}\rfloor-\frac{1}{2}}\right).
   \end{eqnarray} 
Therefore, the result follows from the previous arguments and the bound given by \eqref{iifffff}.

\end{proof}

\noindent
In the following figures:
\begin{itemize}
\item The circle, shown in red $(\textcolor{red}{\rule{0.5cm}{2pt}})$, represents the boundary of the disk containing the numerical range.
\item The circle, shown in green $(\textcolor{green!65!black}{\rule{0.5cm}{2pt}})$, represents the boundary of the disk contained in the numerical range.
\item The Berezin range is indicated by blue ({\huge{\textcolor{blue}{$\bullet$}}}) disk.
\end{itemize}

For the composition-differentiation operator $D_{\phi}$ with symbol $\phi(z)=\frac12z,$ Figure 1 represents the Berezin range, the disk contained in the numerical range, and the disk containing the numerical range.

\begin{figure}[h]
\centering
\includegraphics[width=8cm]{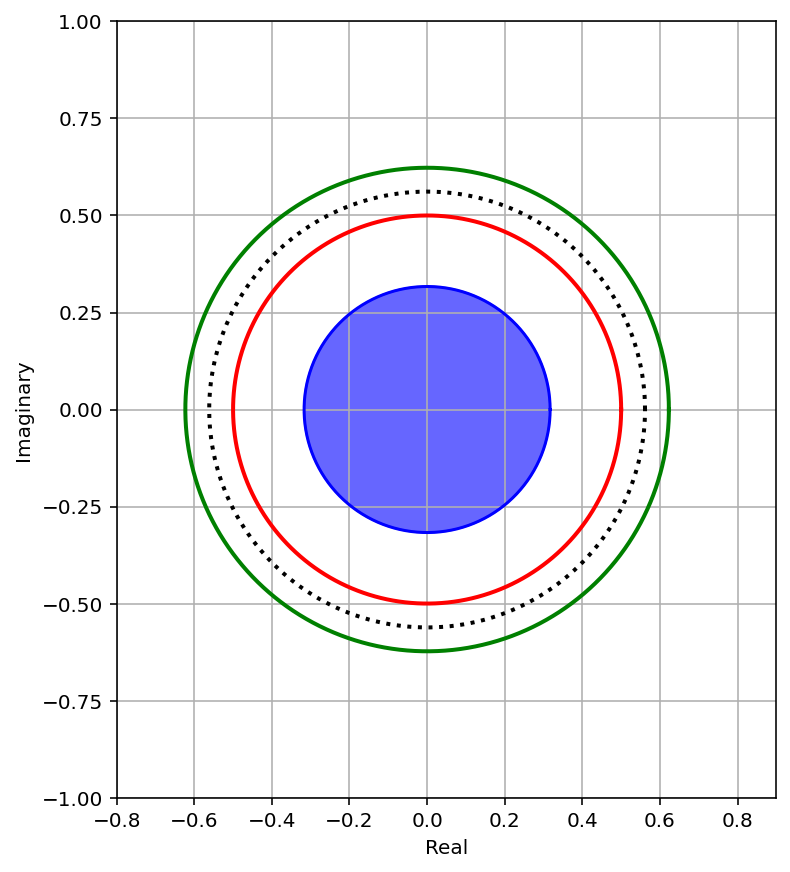}
\caption{}
\label{fig:sample_plot}
\end{figure}

Figure 2 represents the same for the operator $D_{\phi}+0.41I,$ where $\phi(z)=\frac12z.$ Clearly, this operator is not sectorial. However $\textbf{Ber}(D_{\phi}+0.41I)\subseteq S_{\frac{\pi}{3.55}}$ and therefore it is Berezin sectorial with index $\frac{\pi}{3.55}.$

\newpage
 
\begin{figure}[h]
\centering
\includegraphics[width=8cm]{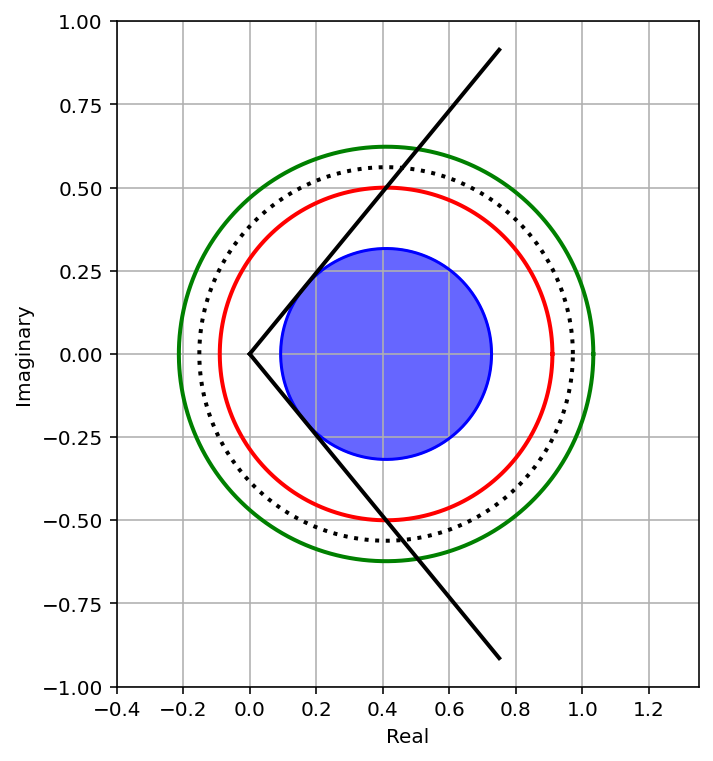}
\caption{}
\label{fig:sample_plot}
\end{figure}

Figure 3 represents the same for the operator $D_{\phi}+0.66I,$ where $\phi(z)=\frac12z.$ Clearly, this operator is sectorial with sectorial index $\frac{\pi}{2.54}$ and is also Berezin sectorial with index $\frac{\pi}{6.27}.$

\begin{figure}[h]
\centering
\includegraphics[width=8cm]{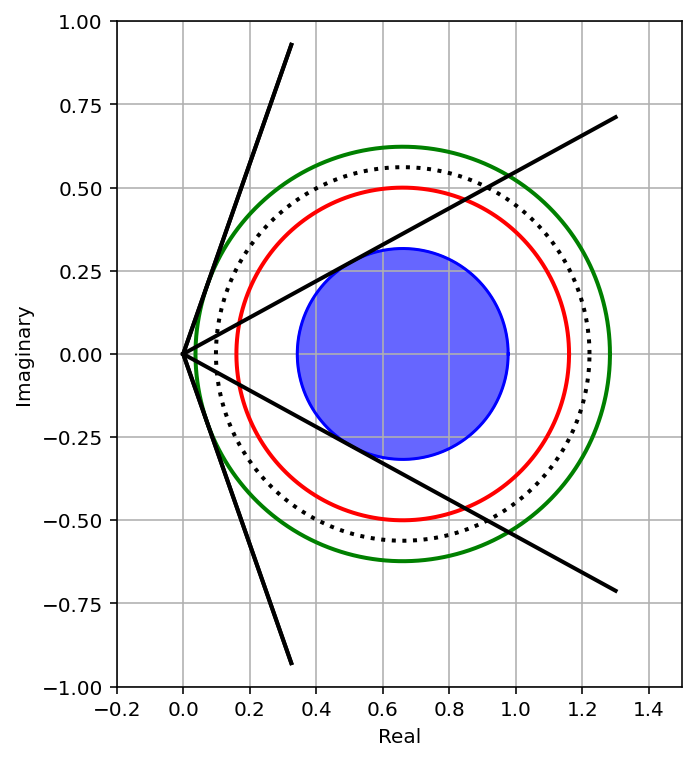}
\caption{}
\label{fig:sample_plot}
\end{figure}

These examples demonstrate that even operators that are not sectorial in the classical sense can be transformed into Berezin sectorial operators. Notably, the Berezin sectorial index is often strictly smaller than the classical sectorial index, as illustrated by the operator $D_{\phi}+0.66I,$ where $\phi(z)=\frac12z.$ As mentioned earlier, in this case the classical sectorial index is $\frac{\pi}{2.54},$ while the Berezin sectorial index reduces to $\frac{\pi}{6.27},$ highlighting the sharper bounds achievable through the Berezin sectorial approach. This observations motivate the study of Berezin number inequalities for Berezin sectorial operators from a geometric perspective and allows to develop Berezin number inequalities to operators that are not classically sectorial.

\section{Inequalities for Berezin sectorial operators}\label{S4}
 
In this section, we derive several inequalities involving the Berezin number for Berezin sectorial operators.
We start with the following lemma, which is used to get our next results.

\begin{lemma}
    Let $T\in \Pi_{\theta}^{\textbf{Ber}}$. Then\label{main}
    \[ \sin\theta\,\textbf{ber}(T)\ge\textbf{ber}(\Im(T)). \]
\end{lemma}
\begin{proof}
 For $\theta=0$, the desired inequality is trivially obtained as the assumption that $T\in \Pi_{\theta}^{\textbf{Ber}}$ implies that $\Im(A) =0$.
 Next, we consider $\theta\neq 0$ and
      $\hat{k}_{\lambda}$ be a normalized reproducing kernel of $\mathscr{H}$. Then we have
     \begin{eqnarray*}
         \big|\langle T\hat{k}_{\lambda},\hat{k}_{\lambda} \rangle \big|& = & \sqrt{\langle \Re (T)\hat{k}_{\lambda},\hat{k}_{\lambda}\rangle^2+\langle \Im (T)\hat{k}_{\lambda},\hat{k}_{\lambda}\rangle^2}\nonumber\\
         & \ge &\sqrt{ \cot^2\theta\,\langle \Im(T)\hat{k}_{\lambda},\hat{k}_{\lambda}\rangle^2+\langle \Im (T)\hat{k}_{\lambda},\hat{k}_{\lambda}\rangle^2}\nonumber\\
          & = & \csc\theta\,\big|\langle \Im(T)\hat{k}_{\lambda},\hat{k}_{\lambda}\rangle\big|.
     \end{eqnarray*}
     It follows that \begin{eqnarray}
        \sin\theta\, \textbf{ber}(T)\ge \big|\langle \Im(T)\hat{k}_{\lambda},\hat{k}_{\lambda}\rangle\big|.\label{bbvv}
     \end{eqnarray}
Therefore, taking supremum over all $\lambda\in \mathscr{X}$, we obtain the desired inequality.
\end{proof}
Now, we obtain a lower bound of the Berezin number.
\begin{theorem}\label{first}
    Let $T\in \Pi_{\theta}^{\textbf{Ber}}$ with $\theta\neq 0$. Then
    \[\textbf{ber}(T)\ge\frac{\csc \theta}{2}\textbf{ber}\left(\Re(T)\pm\Im(T)\right)+\
    \frac{\csc\theta}{2}\Big(\textbf{ber}(\Im(T))-\textbf{ber}(\Re(T))\Big).\]
\end{theorem}
\begin{proof}
    Let $\hat{k}_{\lambda}$ be a normalized reproducing kernel of $\mathscr{H}$. Then, using the Cartesian decomposition of $T$, we have \[
    \big|\langle T\hat{k}_{\lambda},\hat{k}_{\lambda}\rangle\big|^2=\big|\langle \Re(T) \hat{k}_{\lambda},\hat{k}_{\lambda}\rangle\big|^2+\big|\langle \Im(T) \hat{k}_{\lambda},\hat{k}_{\lambda}\rangle\big|^2.
    \]
    From this, we infer that 
    \begin{eqnarray}
        \textbf{ber}(T)\ge \textbf{ber}(\Re(T)).\label{real}
    \end{eqnarray}
    Thus, combining Lemma \ref{main} and inequality (\ref{real}) we obtain 
    \begin{eqnarray*}
        \textbf{ber}(T)&\ge& \max\left\{\textbf{ber}(\Re(T)), \csc\theta\textbf{ber}(\Im(T))\right\}\\
        &=& \frac{1}{2}\left(\textbf{ber}(\Re(T))+ \csc\theta\textbf{ber}(\Im(T))\right)+\frac{1}{2}\left|\textbf{ber}(\Re(T))-\csc\theta\textbf{ber}(\Im(T))\right|\\
        &=& \frac{\csc \theta}{2}\left(\textbf{ber}(\Re(T))+\textbf{ber}(\Im(T))\right)+ \frac{(1-\csc\theta)}{2}\textbf{ber}(\Re(T))\\
    &&\,\,\,\,\,\,\,\,\,\,\,\,\,\,\,\,\,\,\,\,\,\,\,\,\,\,\,\,\,\,\,\,\,\,\,\,\,+\frac{1}{2}\left|\textbf{ber}(\Re(T))-\csc\theta\textbf{ber}(\Im(T))\right|\\
     &\ge& \frac{\csc \theta}{2}\textbf{ber}\big(\Re(T)\pm\Im(T)\big)-\frac{\csc\theta}{2}\textbf{ber}(\Re(T))\\
    &&\,\,\,\,\,\,\,\,\,\,\,\,\,\,\,\,\,\,\,\,\,\,\,\,\,\,\,\,\,\,\,\,\,\,\,\,\,+\frac{1}{2}\left|\textbf{ber}(\Re(T))-\left(\textbf{ber}(\Re(T))-\csc\theta\textbf{ber}(\Im(T))\right)\right|\\
    &=& \frac{\csc \theta}{2}\textbf{ber}\left(\Re(T)\pm\Im(T)\right)+\
    \frac{\csc\theta}{2}\Big(\textbf{ber}(\Im(T))-\textbf{ber}(\Re(T))\Big),
    \end{eqnarray*} as required.
\end{proof}

 Let $T\in \Pi_\theta^{\textbf{Ber}}$  with
 $\Im(T)<0 $. Then, it is clear from the Cartesian decomposition of $T$ that $iT\in \Pi_\theta^{\textbf{Ber}}$ also. Conversely, if $T$ and $iT\in \Pi_\theta^{\textbf{Ber}}$ then $\Im(T)<0$. Let $T\in \mathscr{B(H)}$ be such that
 \begin{eqnarray}
     \textbf{Ber}(T)\subseteq \big\{re^{-i\theta}:\theta_1\le \theta \le\theta_2\big\},\label{spcl}
 \end{eqnarray}
where $r > 0$ and $\theta_1,\theta_2\in(0,\frac{\pi}{2})$. Then, $T$ and $iT\in \Pi_\theta^{\textbf{Ber}}$  with
 sectorial index $\theta_2$ and $\frac{\pi}{2}-\theta_1$, respectively.

\begin{cor}\label{cow}
     Let $T\in \mathscr{B(H)}$ satisfies the property (\ref{spcl}). Then
    \[\textbf{ber}(T)\ge\frac{\csc \theta}{2}\textbf{ber}\left(\Re(T)\pm\Im(T)\right)+\frac{\csc\theta}{2}\left|\textbf{ber}(\Re(T))-\textbf{ber}(\Im(T))\right|,\]
    where $\theta=\max\{\theta_2,\frac{\pi}{2}-\theta_1\}.$
\end{cor}
\begin{proof}
 Since $T$ and $iT$ are in $\Pi_{\theta}^{\textbf{Ber}}$, by using Theorem \ref{first} we get,
 \begin{eqnarray}
     \textbf{ber}(T)\ge\frac{\csc \theta}{2}\textbf{ber}\big(\Re(T)\pm\Im(T)\big)+\frac{\csc\theta}{2}\big(\textbf{ber}(\Im(T))-\textbf{ber}(\Re(T))\big)\label{cat}
 \end{eqnarray}
 and
 \begin{eqnarray}
     \textbf{ber}(T)\ge\frac{\csc \theta}{2}\textbf{ber}\big(\Re(T)\pm\Im(T)\big)+\frac{\csc\theta}{2}\big(\textbf{ber}(\Re(T))-\textbf{ber}(\Im(T))\big).\label{dog}
 \end{eqnarray}
 Thus, by combining inequalities (\ref{cat}) and (\ref{dog}), we obtain the desired inequality.
\end{proof}
\begin{remark} Since $\csc \theta>1$ for any  $\theta\in (0,\frac{\pi}{2})$. Therefore Corollary \ref{cow} gives 
\begin{eqnarray*}
    \textbf{ber}(T)&\ge&\frac{\csc \theta}{2}\textbf{ber}\big(\Re(T)\pm\Im(T)\big)+\frac{\csc\theta}{2}\big|\textbf{ber}(\Re(T))-\textbf{ber}(\Im(T))\big|\\
    &\ge&\frac{1}{2}\big(\textbf{ber}(\Re(T))\pm\textbf{ber}(\Im(T))\big)+\frac{1}{2}\big|\textbf{ber}(\Re(T))-\textbf{ber}(\Im(T))\big|.
\end{eqnarray*}
    Thus, we would like to remark that Corollary \ref{cow} gives a sharper bound than the existing
 bound in \cite[Th. 3.1]{Sen_23}, namely,
 \[
   \textbf{ber}(T)\ge \frac{1}{2}\big(\textbf{ber}(\Re(T))\pm\textbf{ber}(\Im(T))\big)+\frac{1}{2}\big|\textbf{ber}(\Re(T))-\textbf{ber}(\Im(T))\big|.
 \]
\end{remark}

In the next result, we obtain a lower bound for the product of two operators, provided that the product operator is a Berezin sectorial operator with a non-zero sectorial index.

\begin{theorem}\label{kali0}
    Let $S,T\in \mathscr{B(H)}$ and  $T^*S\in \Pi_{\theta}^{\textbf{Ber}}$  with $\theta\neq 0$. Then for any  $\alpha\in \mathbb{R}^+$ 
\begin{eqnarray*}
    \textbf{ber}(T^*S)\ge \max\{\beta_1,\beta_2\},
\end{eqnarray*}
 \begin{eqnarray*}
  \mbox{ where}\,\,\,\,\,\,\,\,\,\, \beta_1 &=&\frac{\csc \theta}{2\alpha} \Big(\||S+i\alpha T|^2\|_{ber}-\| S^*S+\alpha^2 T^*T\|_{ber}\Big),\\
    \mbox{ and}\,\,\,\,\,\,\,\,\,\,  \beta_2 &=&\frac{\csc \theta}{2\alpha} \Big(\| S^*S+\alpha^2 T^*T\|_{ber}-\||S-i\alpha T|^2\|_{ber}\Big).
\end{eqnarray*}
  
\end{theorem}
\begin{proof}  Let $\hat{k}_{\lambda}$ be a normalized reproducing kernel of $\mathscr{H}$. Then we have
    \begin{eqnarray*}
        \| (S+i\alpha T)\hat{k}_{\lambda}\|^2\nonumber
        &=& \langle  S^*S\hat{k}_{\lambda},\hat{k}_{\lambda} \rangle+\alpha^2\langle T^*T\hat{k}_{\lambda},\hat{k}_{\lambda} \rangle+2\alpha\big\langle(\Im(T^*S)\hat{k}_{\lambda},\hat{k}_{\lambda} \big\rangle\nonumber\\
         &=& \langle (S^*S+ \alpha^2 T^*T)\hat{k}_{\lambda},\hat{k}_{\lambda} \rangle+2\alpha\big|\big\langle \Im(T^*S)\hat{k}_{\lambda},\hat{k}_{\lambda} \big\rangle\big|\nonumber\\
         &\le&\| S^*S+ \alpha^2 T^*T\|_{ber}
+2\alpha \sin\theta \textbf{ber}(T^*S).\,\,\Big(\mbox{by inequality (\ref{bbvv})}\Big)
    \end{eqnarray*}

    Hence, taking supremum over all $\lambda\in \Omega$, we get 
\begin{eqnarray}
    \frac{\csc \theta}{2\alpha} \Big(\||S+i\alpha T|^2\|_{ber}-\| S^*S+\alpha^2 T^*T\|_{ber}\Big)\le\textbf{ber}(T^*S)\label{Tss1}
\end{eqnarray}
Again
    \begin{eqnarray*}
       \left \||S-i\alpha T|^2\right\|_{ber}&\ge& \langle |S-i\alpha T|^2\hat{k}_{\lambda}, \hat{k}_{\lambda}\rangle\nonumber\\
   &=&\langle S^*S\hat{k}_{\lambda},\hat{k}_{\lambda}\rangle+\alpha^2\langle
  T^*T\hat{k}_{\lambda},\hat{k}_{\lambda}\rangle- 2\alpha\Im\langle T^*S\hat{k}_{\lambda}, \hat{k}_{\lambda}\rangle\nonumber\\
    &\ge&\langle (S^*S+\alpha^2T^*T)\hat{k}_{\lambda},\hat{k}_{\lambda}\rangle- 2\alpha|\Im\langle S^*T\hat{k}_{\lambda}, \hat{k}_{\lambda}\rangle|\\
     &\ge&\langle (S^*S+\alpha^2T^*T)\hat{k}_{\lambda},\hat{k}_{\lambda}\rangle- 2\alpha\sin\theta\textbf{ber}(T^*S).\,\,\Big(\mbox{by inequality (\ref{bbvv})}\Big)\end{eqnarray*}
  Thus, taking supremum over all $\lambda\in \Omega$, we get 
       \begin{eqnarray}
             \frac{\csc \theta}{2\alpha} \Big(\| S^*S+\alpha^2 T^*T\|_{ber}-\||S-i\alpha T|^2\|_{ber}\Big)\le\textbf{ber}(T^*S)\label{Tss3}
       \end{eqnarray}
Finally, combining inequalities (\ref{Tss1}) and (\ref{Tss3}), we arrive at the required result.
\end{proof}

The following two inequalities are immediate from Theorem \ref{kali0}  by taking $T=S^*$ and $T=I$, respectively.
\begin{cor}
(i) If $S^2\in \Pi_{\theta}^{\textbf{ber }}$ with $\theta\neq 0$, then \[\textbf{ber}(S^2)\ge \max\{\beta_1 ,\beta_2 \},\]  \begin{eqnarray*}
       \mbox{ where}\,\,\,\,\,\,\,\,\,\, \beta_1 &=& \frac{\csc \theta}{2\alpha} \Big(\||S+i\alpha S^*|^2\|_{ber}-\| S^*S+\alpha^2 SS^*\|_{ber}\Big),\\
         \mbox{ and}\,\,\,\,\,\,\,\,\,\,\beta_2 &=&\frac{\csc \theta}{2\alpha} \Big(\| S^*S+\alpha^2 SS^*\|_{ber}-\||S-i\alpha S^*|^2\|_{ber}\Big).
    \end{eqnarray*}\\
  
  (ii)  If $S\in \Pi_{\theta}^{\textbf{ber }}$ with $\theta\neq 0$, then \[\textbf{ber}(S)\ge \max\{\beta_1 ,\beta_2 \},\]  \begin{eqnarray*}
      \mbox{ where}\,\,\,\,\,\,\,\,\,\,  \beta_1 &=& \frac{\csc \theta}{2\alpha} \Big(\||S+i\alpha I|^2\|_{ber}-\| S^*S+\alpha^2 I\|_{ber}\Big),\\
      \mbox{ and}\,\,\,\,\,\,\,\,\,\, \beta_2 &=&\frac{\csc \theta}{2\alpha} \Big(\| S^*S+\alpha^2 I\|_{ber}-\||S-i\alpha I|^2\|_{ber}\Big).
    \end{eqnarray*}

\end{cor}





\begin{remark}
  $(i)$ According to \cite[Cor. 2.24]{pintuani}, the inequality
 \begin{eqnarray}
 \frac{1}{2}   \Big( \| S+i T\|^2_{\textbf{ber}}-\| S^*S+T^*T\|_{\textbf{ber}}\Big)\le \textbf{ber}(T^*S)\label{kali1}.
 \end{eqnarray}
holds. Note that $\|S+i\alpha T \|_{ber}^2\le \||s+i\alpha T|^2\|_{ber}$ and  $\csc \theta >1 $  for any $\theta\in(0,\frac{\pi}{2})$.  Thus,  for $\alpha=1$, Theorem \ref{kali0} improves the  bound of (\ref{kali1}).\\

 $(ii)$  Taking $A=S$ and $B=iT$ in \cite[ Th. 5.1]{garayevlama}, we obtain
\begin{eqnarray}
      \frac{1}{2}  \textbf{ber}(S^*S+T^*T)-\frac{r^2}{2}
     \le \textbf{ber}(T^*S),\label{sad1}
\end{eqnarray} where $\|S-iT\|\le r$. Note that  $\||T-iS|^2\|_{ber}\le \|T-iS\|^2$  and $\cos\theta > 1$ for any $\theta\in(0,\frac{\pi}{2})$. Hence, for $\alpha=1$, the bound of Theorem \ref{kali0} refines the bound in  (\ref{sad1}).\\

$(iii)$ According to \cite[Th 3.2]{pin2}, the following inequality holds:
  \begin{eqnarray}
        \| S+i T\|^2_{\textbf{ber}}-\Big(\|S\|_{ber}^2+\|T\|_{ber}^2+\textbf{ber}^{\frac{1}{2}}(S^*S)\textbf{ber}^{\frac{1}{2}}(T^*T)\Big)\le \textbf{ber}(T^*S).\label{ppsskk}
  \end{eqnarray}

  Let $S=\begin{pmatrix} 1 & 0\\
  0 & 3\end{pmatrix} $ and $T=\begin{pmatrix} \frac{\sqrt{2}-i\sqrt{2}}{2}& 0\\
 0 & \sqrt{2}-i\sqrt{2} \end{pmatrix} $ acting on $\mathscr{H}=\mathbb{C}^2$. Note that  $T^*S\in\Pi_{\frac{\pi}{4}}^{\textbf{Ber}}$.
Then inequality (\ref{ppsskk}) gives the lower bound
 $ 6(\sqrt{2}-1)\le \textbf{ber}(T^*S) $, whereas Theorem \ref{kali0} provides the sharper estimate  $ 6\le\textbf{ber}(T^*S)$. Thus, Theorem \ref{kali0} yields a strictly better lower bound than (\ref{ppsskk}).
\end{remark}




 In the following theorem, we prove inequalities involving the Berezin number and operator norm. We assume that the operator $B:\mathscr{H}\to \mathscr{H}$ is an invertible bounded linear operator with bounded inverse $B^{-1}$. Then it is easy to observe that
\[\sqrt{\langle B^*Bx,x\rangle}\ge \|B^{-1}\|^{-1}\sqrt{\langle x,x \rangle}~~~~\text{for all $x\in \mathscr{H}$.}\]
\begin{theorem}\label{sad4}
    Let $T,S\in \mathscr{B(H)}$ and $S^*T\in \Pi_{\theta}^{\textbf{ber }}$. If $S$ is invertible, then
    \[\||T|^2\|_{ber}^{\frac{1}{2}}\le\|S^{-1}\| \Big(\sin\theta\textbf{ber}( S^*T)+  \frac{1}{2}\left\||T-iS|^2\right\|_{ber}\Big).
    \]
\end{theorem}

\begin{proof} Let $\hat{k}_{\lambda}$ be a normalized reproducing kernel of $\mathscr{H}$.
    Then  we have
 \begin{eqnarray}
       \left \||T-iS|^2\right\|_{ber}&\ge& \langle |T-iS|^2\hat{k}_{\lambda}, \hat{k}_{\lambda}\rangle\nonumber\\
   &=&\langle T^*T\hat{k}_{\lambda},\hat{k}_{\lambda}\rangle+\langle
  S^*S\hat{k}_{\lambda},\hat{k}_{\lambda}\rangle- 2\Im\langle S^*T\hat{k}_{\lambda}, \hat{k}_{\lambda}\rangle\nonumber\\
    &\ge&\langle T^*T\hat{k}_{\lambda},\hat{k}_{\lambda}\rangle+\langle 
  S^*S\hat{k}_{\lambda},\hat{k}_{\lambda}\rangle- 2|\Im\langle S^*T\hat{k}_{\lambda}, \hat{k}_{\lambda}\rangle|.\label{sad4}
    \end{eqnarray}

    Since, $S$ is invertible so $\sqrt{\langle S^*S\hat{k}_{\lambda},\hat{k}_{\lambda}\rangle}\ge \|S^{-1}\|^{-1}$. This implies that 
    \begin{eqnarray}
         2\sqrt{\langle T^*T\hat{k}_{\lambda},\hat{k}_{\lambda}\rangle}\|S^{-1}\|^{-1}&\le& \langle T^*T\hat{k}_{\lambda},\hat{k}_{\lambda}\rangle+\|S^{-1}\|^{-2}\nonumber\\
         &\le& \langle T^*T\hat{k}_{\lambda},\hat{k}_{\lambda}\rangle+\langle S^*S \hat{k}_{\lambda},\hat{k}_{\lambda}\rangle.\label{sad3}
    \end{eqnarray}
    Now, it follows from the inequalities (\ref{sad4}) and (\ref{sad3}) that 
    \begin{eqnarray*}
        \sqrt{\langle |T|^2\hat{k}_{\lambda},\hat{k}_{\lambda}\rangle}&\le& \|S^{-1}\| \Big( |\Im\langle S^*T\hat{k}_{\lambda}, \hat{k}_{\lambda}\rangle|+  \frac{1}{2}\||T-iS|^2\|_{ber}\Big)\\
        &\le& \|S^{-1}\| \Big(\sin\theta\textbf{ber}( S^*T)+  \frac{1}{2}\||T-iS|^2\|_{ber}\Big).
    \end{eqnarray*}   
    Thus, by taking the supremum over all
    $\lambda\in \Omega$, we obtain the required bound.
\end{proof}
\begin{remark}
    By taking $B=iS$ in \cite[ Th. 2.20]{vietnamani},  the inequality becomes
    \begin{eqnarray}
        \||T|^2\|_{ber}^{\frac{1}{2}}\le\|S^{-1}\| \Big( \textbf{ber}( S^*T)+  \frac{1}{2}\||T-iS|^2\|_{ber}\Big).
    \label{sad5}\end{eqnarray} 
   Since $\sin\theta<1$ for $\theta\in [0,\frac{\pi}{2})$ so Theorem \ref{sad4} yields a better bound than the bound  given in (\ref{sad5}).
\end{remark}
By choosing $S=T^*$ in Theorem \ref{sad4}, we obtain the following corollary.
\begin{cor}
    Let $T\in \mathscr{B(H)}$ and
 $T^2\in \Pi_{\theta}^{\textbf{ber }}$, then 
 \[\||T|^2\|_{ber}^{\frac{1}{2}}\le\|T^{-1}\| \Big(\sin\theta\textbf{ber}( T^2)+  \frac{1}{2}\||T-iT^*|^2\|_{ber}\Big).\]
\end{cor}

\begin{lemma}\cite{potpourri}\label{last30}
    Let $(\mathscr{H};\langle \cdot,\cdot\rangle)$ be a complex inner product space and $x,y\in \mathscr{H}$. Then for any $t\in \mathbb{R}$ the following inequality holds:
    \begin{eqnarray*}
      0\le \|x\|^2\|y\|^2\le\frac{1}{4}\Big(\Im\langle x,y \rangle+\Re\langle x,y \rangle\Big)^2+  \frac{\|y\|^2}{2}\Big(\|x-ty\|^2+\|x-ity\|^2\Big).
    \end{eqnarray*}
\end{lemma}
\begin{theorem}\label{last31} Let $T\in\Pi_{\theta}^{\textbf{Ber}}$ and $t\in \mathbb{R}$. Then
   \[ \||T|^2\|_{ber}\le \frac{1}{4}(\sin\theta+1)^2\textbf{ber}^2(T)+\frac{1}{2}\inf_{t\in \mathbb{R}}\Big(\left\||T-tI|^2\right\|_{ber}+\left\||T-itI|^2\right\|_{ber}\Big).\]
\end{theorem}
\begin{proof}
    Let $\hat{k}_{\lambda}$ be a normalized reproducing kernel of $\mathscr{H}$ and $t\in \mathbb{R}$. Now, choosing $x=T\hat{k}_{\lambda}$ and $y=\hat{k}_{\lambda}$ in Lemma \ref{last30}, we get 
      \begin{eqnarray*}
\|T\hat{k}_{\lambda}\|^2&\le& \frac{1}{4}\Big(\Im\langle T\hat{k}_{\lambda},\hat{k}_{\lambda} \rangle+\Re\langle T\hat{k}_{\lambda},\hat{k}_{\lambda} \rangle\Big)^2+\frac{1}{2}\Big(\|T\hat{k}_{\lambda}-t\hat{k}_{\lambda}\|^2+\|T\hat{k}_{\lambda}-it\hat{k}_{\lambda}\|^2\Big)\\
&\le& \frac{1}{4}\Big(|\Im\langle T\hat{k}_{\lambda},\hat{k}_{\lambda} \rangle|+|\Re\langle T\hat{k}_{\lambda},\hat{k}_{\lambda} \rangle|\Big)^2+  \frac{1}{2}\Big(\||T-tI|^2\|_{ber}+\||T-itI|^2\|_{ber}\Big)\\
&\le& \frac{1}{4}\big(\sin\theta+1\big)^2\textbf{ber}^2(T)+  \frac{1}{2}\Big(\||T-tI|^2\|_{ber}+\||T-itI|^2\|_{ber}\Big).
    \end{eqnarray*}
   Taking supremum over $\lambda\in \mathscr{X}$, we get \begin{eqnarray*}
   \||T|^2\|_{ber}\le \frac{1}{4}(\sin\theta+1)^2\textbf{ber}^2(T)+\frac{1}{2}\Big(\left\||T-tI|^2\right\|_{ber}+\left\||T-itI|^2\right\|_{ber}\Big).    
   \end{eqnarray*} 
   Since the  above inequality holds for all
   $t\in\mathbb{R}$, the required result follows by taking the infimum over $t\in \mathbb{R}$.
\end{proof}

\begin{remark}
    In \cite[Cor 4.1]{garayevlama} it was derived that 
    \begin{eqnarray}
         \||T|^2\|_{ber}\le \frac{1}{2}\textbf{ber}^2(T)+\frac{1}{2}\inf_{t\in \mathbb{R}}\Big(\||T-tI|^2\|_{ber}+\||T-itI|^2\|_{ber}\Big).\label{last32}
    \end{eqnarray} Clearly, if $\theta< \sin^{-1}(\sqrt{2}-1)$ then Theorem \ref{last31}  provides a better bound than that in  (\ref{last32}).
\end{remark}

\section{Power inequalities and Berezin sectorial operators}\label{S5}
For any bounded linear operator $T$, the numerical radius satisfies the power inequality (\cite{halmos}), namely,
 $$w(T^n)\le w(T)^n\,\,\,(n\in \mathbb{N}).$$
 Consequently, for any bounded operator 
 $T$ acting on a reproducing kernel Hilbert space $\mathscr{H}$, we have 
 $$\mbox{\textbf{ber}}(T^n)\le\left(\frac{w(T)}{\mbox{\textbf{ber}}(T)}\right)^n\mbox{\textbf{ber}}^n(T)\,\,\,(n\in \mathbb N).$$
A natural question arising from this is whether the inequality
 $$\mbox{\textbf{ber}}(T^n)\le\mbox{\textbf{ber}}^n(T)\,\,\,(n\in \mathbb N) ,$$ holds in general?

 In \cite[Th. 3]{coburn}, author constructed certain unitary operators on the weighted Bergman spaces and used these to show that the inequality 
$\mbox{\textbf{ber}}(T^2)\le\mbox{\textbf{ber}}^2(T)$ does not hold, and further he showed that it even satisfies the opposite inequality $\mbox{\textbf{ber}}(T^2)>\mbox{\textbf{ber}}^2(T).$
Furthermore, as seen in \cite{garayevlama}, some specific operators belonging to a norm-closed subalgebra also satisfy this power inequality.  In \cite{garayev}, it is also proved that this inequality holds for Toeplitz operators on the Hardy Hilbert space and Bergman space under certain conditions.

In our next resut we prove the power inequality of the Toeplitz operator on the weighted Bergman space for some special conditions. For this, first we recall that
for $\alpha>-1$, $A_\alpha^2(\mathbb{D})$ denotes the weighted Bergman space on the open unit disk consisting of all holomorphic functions such that $\int_{\mathbb{D}}|f(z)|^2dA_{\alpha}(z)<\infty,$ where $dA$ denotes the normalised Lebesgue measure on $\mathbb{D}$ and $dA_{\alpha}(z)=(\alpha+1)(1-|z|^2)^\alpha dA(z).$ The reproducing kernel of the weighted Bergman space at the point $w\in\mathbb{D}$ is $$k^\alpha_w(z)=\frac{1}{(1-\overline{w}z)^{\alpha+2}}.$$\\
Let $P_\alpha$ be the orthogonal projection of $L^2(\mathbb{D},dA_\alpha)$ onto $A_\alpha^2(\mathbb{D}).$
Suppose $L^\infty(\mathbb{D},dA_\alpha)$ is the space of all complex measurable functions $\phi$ on $\mathbb{D}$ such that $$\|\phi\|_{\infty,\alpha}=\sup\left\{c\ge0:A_\alpha(\{z\in\mathbb{D}:|\phi(z)|>c\})>0\right\}<\infty.$$
For $\phi\in L^\infty(\mathbb{D},dA_\alpha),$ the operator $T_\phi$ on $A_{\alpha}^2(\mathbb{D})$ defined by $$T_\phi f=P_\alpha(\phi f),\,\, f\in A_\alpha^2$$ is called Toeplitz operator on $A_\alpha^2$ with symbol $\phi.$
Clearly, $T_\phi$ is a bounded linear operator on $L_\alpha^2(dA_\alpha)$ with $\|T_\phi\|\le\|\phi\|_{\infty,\alpha},$ see \cite{kzhu}.

\begin{proposition} \label{iixx}
    Let $\phi$ be harmonic on $\mathbb{D}$ and $T_{\phi}$ be a Toeplitz operator on $A_{\alpha}^2(\mathbb{D})$. Then for any integer $n\ge 1$, the following inequality holds
    \[\textbf{ber}(T^n_{\phi})\le \textbf{ber}^n(T_{\phi}).\]
\end{proposition}
\begin{proof}
   
    Since $\phi$ is harmonic on $\mathbb{D},$ then from \cite[Th. 2.1]{berezin toeplitz}, we have
    $ \widetilde{T_{\phi}}(w)=\phi(w)$ for $w\in\mathbb{D}$ and hence $\textbf{ber}(T_{\phi})=\|\phi\|_{\infty}$.\\
    Now, \[ \textbf{ber}(T_{\phi}^n)\le \|T_{\phi}^n\|\le \|T_{\phi}\|^n=\|\phi\|^n_{\infty}=\left(\sup_{\lambda\in \mathbb{D}}| \widetilde{\phi}(\lambda)|\right)^n=\textbf{ber}(T_{\phi})^n.\]
    So, for any $n\ge 1$ we get $\textbf{ber}(T^n_{\phi})\le \textbf{ber}^n(T_{\phi})$.
\end{proof}

\begin{remark}\label{Re}
    If $\phi$ is harmonic on $\mathbb{D}$ and $T_{\phi}$ is a Toeplitz operator on $A_\alpha^2(\mathbb{D}),$ then we have $\Re(\phi)$ is harmonic and the real part of $T_{\phi}$ is also a Toeplitz operator on $A_\alpha^2(\mathbb{D})$. Hence by Proposition \ref{iixx}, we obtain  
    \begin{eqnarray*}
        \textbf{ber}^n\left(\Re(T_{\phi})\right)\le \textbf{ber}^n\left(\Re(T_{\phi})\right).
    \end{eqnarray*}for any integer $n\ge 1$. Similarly, the same argument follows for the imaginary part of $T_{\phi}.$
\end{remark}

Here, we establish several inequalities for the Berezin number and Berezin norm of Berezin sectorial operators whose real and imaginary parts satisfy the power inequality of the Berezin number.
The collection of all such operators is denoted by $\Pi_{\theta}^{\textbf{Ber},P}$ and is defined as
 \begin{eqnarray*}
\Pi_{\theta}^{\textbf{Ber},P}
&=&\Big\{ T\in \Pi_{\theta}^{\textbf{ber }} : \textbf{ber}\left(\Re^n(T)\right)\le \textbf{ber}^n\left (\Re(T)\right),\\
&&\,\,\,\,\,\,\,\,\,\,\textbf{ber}\left(\Im^n(T)\right)\le \textbf{ber}^n \left(\Im(T)\right)\,\, \forall n\in\mathbb{N}\Big\}.     
 \end{eqnarray*}
Clearly, if $\phi$ is harmonic on $\mathbb{D}$ and $M \in \mathbb R$ satisfies $\|\phi\|_{\infty}<M,$ then 
$T_{\phi+M}\in \Pi_{\theta}^{\textbf{Ber},P}.$
\begin{lemma}

\label{a*a+aa*}
    Let $T\in \Pi_{\theta}^{\textbf{Ber},P}$ with $\theta \neq 0$. Then 
    \[  \textbf{ber}^2(T)\ge \frac{\csc^2 \theta}{4}\left\|T^*T+TT^*\right\|_{ber}+\frac{\csc^2\theta}{2}\Big(\textbf{ber}^2(\Im(T))-\textbf{ber}^2(\Re(T))\Big).\]
\end{lemma}
\begin{proof}    Let $T\in \Pi_{\theta}^{\textbf{Ber},P}$. Then from Lemma \ref{main} we get
      \begin{eqnarray*}
        &&\textbf{ber}^2(T)\\
        &\ge& \max\left\{\textbf{ber}^2(\Re(T)), \csc^2\theta\textbf{ber}^2(\Im(T))\right\}\\
        &=& \frac{\textbf{ber}^2(\Re(T))+ \csc^2\theta\textbf{ber}^2(\Im(T))}{2}+\frac{\left|\textbf{ber}^2(\Re(T))-\csc^2\theta\textbf{ber}^2(\Im(T))\right|}{2}\\
        &=& \frac{\csc^2 \theta}{2}\left(\textbf{ber}^2(\Re(T))+\textbf{ber}^2(\Im(T))\right)+ \frac{(1-\csc\theta)}{2}\textbf{ber}^2(\Re(T))\\
&&\,\,\,\,\,\,\,\,\,\,\,\,\,\,\,\,\,\,\,\,\,\,\,\,\,\,\,\,\,\,\,\,\,\,\,\,\,+\frac{1}{2}\left|\textbf{ber}^2(\Re(T))-\csc\theta\textbf{ber}^2(\Im(T))\right|\\
 &\ge& \frac{\csc^2 \theta}{2}\Big(\textbf{ber}(\Re^2(T))+\textbf{ber}(\Im^2(T))\Big)+ \frac{(1-\csc\theta)}{2}\textbf{ber}^2(\Re(T))\\
&&\,\,\,\,\,\,\,\,\,\,\,\,\,\,\,\,\,\,\,\,\,\,\,\,\,\,\,\,\,\,\,\,\,\,\,\,\,+\frac{1}{2}\left|\textbf{ber}^2(\Re(T))-\csc\theta\textbf{ber}^2(\Im(T))\right|\\
     &\ge& \frac{\csc^2 \theta}{2}\textbf{ber}(\Re^2(T)+\Im^2(T))-\frac{\csc^2\theta}{2}\textbf{ber}^2(\Re(T))\\
&&\,\,\,\,\,\,\,\,\,\,\,\,\,\,\,\,\,\,\,\,\,\,\,\,\,\,\,\,\,\,\,\,\,\,\,\,\,+\frac{1}{2}\left|\textbf{ber}^2(\Re(T))-\big(\textbf{ber}^2(\Re(T))-\csc^2\theta\textbf{ber}^2(\Im(T))\big)\right|\\
    &=& \frac{\csc^2 \theta}{4}\left\|T^*T+TT^*\right\|_{ber}+\frac{\csc^2\theta}{2}\Big(\textbf{ber}^2(\Im(T))-\textbf{ber}^2(\Re(T))\Big).
    \end{eqnarray*} 
    This completes the proof.
\end{proof} 
By applying Lemma \ref{a*a+aa*}, we establish upper bounds for the Berezin number of generalized commutator and anti-commutator operators. Recall that the operator $ST-TS$
 represents the commutator, while 
$ST+TS$ represents the anti-commutator.
We now observe the following scalar  inequality
\begin{eqnarray}
    \big(ab+cd\big)^2&\le &  \big(a^2+c^2 \big) \big(b^2+d^2 \big),\,\,\,\,\,\mbox{ where $a,b,c,d\in\mathbb{R}$.}\label{ab+cd}
\end{eqnarray}
Using this inequality, the following result is obtained.

\begin{theorem}\label{axb+bya}
  Let $S,T,X,Y\in \mathscr{B(H)}$  where $S\in \Pi_{\theta}^{\textbf{Ber},P}$  with $\theta \neq 0$. Then
  \begin{eqnarray*}
      &&\textbf{ber}(SXT\pm TYS )\\
     &\le& 2\sin \theta\left\|T^*|X|^2T+T|Y^*|^2T^*\right\|_{ber}^{\frac{1}{2}} \left( \textbf{ber}^2(S)-\frac{\csc^2\theta}{2}\left(\textbf{ber}^2(\Im(S))-\textbf{ber}^2(\Re(S))\right)\right)^{\frac{1}{2}}.
  \end{eqnarray*}
\end{theorem}
\begin{proof}

  Let $\hat{k}_{\lambda}$ be a normalized reproducing kernel of $\mathscr{H}$. Then we have
\begin{eqnarray}
    \big|\langle (SXT\pm TYS ) \hat{k}_{\lambda},\hat{k}_{\lambda} \rangle\big|&\le&|\langle SXT \hat{k}_{\lambda},\hat{k}_{\lambda}\rangle|+|\langle T YS \hat{k}_{\lambda},\hat{k}_{\lambda} \rangle|\nonumber \label{gpl01}\\
    &=& |\langle XT \hat{k}_{\lambda},S^*\hat{k}_{\lambda}\rangle|+|\langle S \hat{k}_{\lambda},(TY)^*\hat{k}_{\lambda} \rangle|.\end{eqnarray}
   Now, by applying the Cauchy-Schwarz inequality together with the inequality (\ref{ab+cd}), we obtain
    \begin{eqnarray}
    &&|\langle XT \hat{k}_{\lambda},S^*\hat{k}_{\lambda}\rangle|+|\langle S \hat{k}_{\lambda},(TY)^*\hat{k}_{\lambda} \rangle|\nonumber\\
    &&\le\langle SS^* \hat{k}_{\lambda},\hat{k}_{\lambda} \rangle^{\frac{1}{2}} \langle T^*|X|^2T\hat{k}_{\lambda},\hat{k}_{\lambda} \rangle^{\frac{1}{2}} +\langle S^*S \hat{k}_{\lambda},\hat{k}_{\lambda} \rangle^{\frac{1}{2}} \langle T|Y^*|^2T^* \hat{k}_{\lambda},\hat{k}_{\lambda} \rangle^{\frac{1}{2}}\nonumber \\
    &&\le \big\langle (S^*S+SS^*)\hat{k}_{\lambda},\hat{k}_{\lambda}\big\rangle^{\frac{1}{2}}\big\langle (T^*|X|^2T+T|Y^*|^2T^*)\hat{k}_{\lambda},\hat{k}_{\lambda}\big\rangle^{\frac{1}{2}}\nonumber\\
    &&\le\left \|S^*S+SS^*\right\|_{ber}^{\frac{1}{2}}\left\|T^*|X|^2T+T|Y^*|^2T^*\right\|_{ber}^{\frac{1}{2}}\nonumber\\
    &&\le 2\sin \theta\left\|T^*|X|^2T+T|Y^*|^2T^*\right\|_{ber}^{\frac{1}{2}}\nonumber\\
    &&\Big( \textbf{ber}^2(S)-\frac{\csc^2\theta}{2}\left(\textbf{ber}^2(\Im(S))-\textbf{ber}^2(\Re(S))\right)\Big)^{\frac{1}{2}}\,\big(\mbox{by  Lemma \ref{a*a+aa*}}\big).\label{gpl02}
\end{eqnarray} Hence, from the inequalities (\ref{gpl01}) and (\ref{gpl02}) we have 
\begin{eqnarray*}
        \big|\langle (SXT\pm TYS ) \hat{k}_{\lambda},\hat{k}_{\lambda} \rangle\big|&\le&2\sin \theta\left\|T^*|X|^2T+T|Y^*|^2T^*\right\|_{ber}^{\frac{1}{2}} \nonumber\\
    &&\Big( \textbf{ber}^2(S)-\frac{\csc^2\theta}{2}\left(\textbf{ber}^2(\Im(S))-\textbf{ber}^2(\Re(S))\right)\Big)^{\frac{1}{2}}.
\end{eqnarray*}
   By taking the supremum over all $\lambda\in\mathscr{X}$, the desired inequality is obtained.
\end{proof}
The following corollary is obtained from Theorem \ref{axb+bya} by setting $X=Y=I$.
\begin{cor}\label{ab+ban1}
     Let $S,T\in \mathscr{B(H)}$ where $S\in \Pi_{\theta}^{\textbf{Ber},P}$  with $\theta \neq 0$. Then
  \begin{eqnarray*}
      &&\textbf{ber}(ST\pm TS )\\
     &\le&2\sin \theta\left\|T^*T+TT^*\right\|_{ber}^{\frac{1}{2}} \left( \textbf{ber}^2(S)-\frac{\csc^2\theta}{2}\left(\textbf{ber}^2(\Im(S))-\textbf{ber}^2(\Re(S))\right)\right)^{\frac{1}{2}}.
  \end{eqnarray*}
\end{cor}
\begin{remark}

 In \cite[Th. 4]{blockpintu}, the following inequality was established:
 \begin{eqnarray*}
     \textbf{ber}^p(AB^*\pm CD^*)\le \frac{1}{2}\left\|(AA^*+CC^*)^p+(BB^*+DD^*)^p\right\|_{ber},\,\,\,\forall\,\, p\ge 1,
 \end{eqnarray*} where $A,B,C,D\in\mathscr{B(H)}$. Now, if we consider $A=S$, $B=T^*$, $C=T$, $D=S^*$ and $p=1$ in the above inequality then it becomes
 \begin{eqnarray}
     \textbf{ber}(ST\pm TS)
    & \le& \frac{1}{2}\left\|T^*T+S^*S+TT^*+SS^*\right\|_{ber}.\label{example}
 \end{eqnarray}
  Consider $S=\begin{pmatrix}
    1+i0.2 & 0\\
    0 & 2+i0.5
\end{pmatrix}$ and $T=\begin{pmatrix}
    0.7 & 0\\
    0 & 0
\end{pmatrix}$. It is clear that $S\in \Pi_{\frac{\pi}{12}}^{\textbf{Ber},P} $. Then using the  inequality (\ref{example}), we obtain $  \textbf{ber}(ST\pm TS )\le 4.25$,
whereas Corollary \ref{ab+ban1} provides a sharper estimate, namely $  \textbf{ber}(ST\pm TS )\le 2.91$. Therefore, Corollary \ref{ab+ban1} yields a tighter bound than inequality (\ref{example}).
\end{remark}
Applying Corollary \ref{ab+ban1}, we arrive at the following result.

\begin{cor}
      Let $S,T\in  \Pi_{\theta}^{\textbf{Ber},P}$  with $\theta \neq 0$. Then
      \begin{eqnarray*}
          \textbf{ber}(ST\pm TS )\le \min\{\beta_1,\beta_2\},
      \end{eqnarray*}
      where \begin{eqnarray*}
          \beta_1 \le 2\sin \theta\left\|T^*T+TT^*\right\|_{ber}^{\frac{1}{2}}\left( \textbf{ber}^2(S)-\frac{\csc^2\theta}{2}\left(\|\Im(S)\|_{ber}^2-\|\Re(S)\|_{ber}^2\right)\right)^{\frac{1}{2}}
      \end{eqnarray*}
      and \begin{eqnarray*}
          \beta_2 \le 2\sin \theta\left\|S^*S+SS^*\right\|_{ber}^{\frac{1}{2}} \left( \textbf{ber}^2(T)-\frac{\csc^2\theta}{2}\left(\|\Im(T)\|_{ber}^2-\|\Re(T)\|_{ber}^2\right)\right)^{\frac{1}{2}}.
      \end{eqnarray*}
\end{cor}
\begin{proof}
    The proof follows directly from Corollary \ref{ab+ban1} by interchanging $S$ and $T$.
\end{proof}
Next, we derive Berezin number inequalities for the sum and product of commuting operators in 
$\Pi_{\theta}^{\textbf{Ber},P}$. To begin, we recall the following scalar inequality
\begin{eqnarray}
    \left(\sum_{j=1}^n a_jb_j\right)^2\le\left(\sum_{i=1}^n a_j^2\right)\left(\sum_{i=1}^n b_j^2\right),\,\,\,\,\,\mbox{ where $a_j,b_j\in\mathbb{R}$.}\label{gencauchy}
\end{eqnarray}
 Applying the above inequalities, we prove the following theorem.

\begin{theorem}\label{commutating} Let $S_j,T_j\in \Pi_{\theta}^{\textbf{Ber},P}$, $j=1,2,\cdots n$  be such that $S_jT_j=T_jS_j$. Then
   \begin{eqnarray*}
       \textbf{ber}^2\left(\sum_{j=1}^{n} S_jT_j\right)\le\left (1+\sin^2\theta\right)^2\left(\sum_{j=1}^n\textbf{ber}^2(S_j)\right)\left(\sum_{j=1}^n\textbf{ber}^2(T_j)\right).
   \end{eqnarray*}
\end{theorem}

\begin{proof}  Let $\hat{k}_{\lambda}$ be a normalized reproducing kernel of $\mathscr{H}$. Then using the inequalities (\ref{ab+cd}) and (\ref{gencauchy}), we get
    \begin{eqnarray*}&&
   \left|\left\langle \sum_{j=1}^n S_jT_j\hat{k}_{\lambda},\hat{k}_{\lambda}\right\rangle\right|^2\\
         &=& \frac{1}{4} \left|\left\langle\left(\sum_{j=1}^nS_jT_j+T_jS_j\right) \hat{k}_{\lambda},\hat{k}_{\lambda}\right\rangle\right|^2\\
         &\le &  \frac{1}{4} \left(\sum_{j=1}^n\left(\left|\langle S_jT_j \hat{k}_{\lambda},\hat{k}_{\lambda}\rangle\right|+\left|\langle T_jS_j \hat{k}_{\lambda},\hat{k}_{\lambda}\rangle\right|\right)\right)^2\\
           &\le &  \frac{1}{4} \left(\sum_{j=1}^n\left(\langle T_j^*T_j \hat{k}_{\lambda},\hat{k}_{\lambda}\rangle^{\frac{1}{2}} \langle S_jS_j^*\hat{k}_{\lambda},\hat{k}_{\lambda}\rangle^{\frac{1}{2}}+\langle S_j^*S_j \hat{k}_{\lambda},\hat{k}_{\lambda}\rangle^{\frac{1}{2}} \langle T_jT_j^* \hat{k}_{\lambda},\hat{k}_{\lambda}\rangle^{\frac{1}{2}}\right)\right)^2\\
              &\le &  \frac{1}{4}\left(\sum_{j=1}^n \left(\langle S_j^*S_j \hat{k}_{\lambda},\hat{k}_{\lambda}\rangle+\langle S_jS_j^*\hat{k}_{\lambda},\hat{k}_{\lambda}\rangle\right)^{\frac{1}{2}}\left(\langle T_j^*T_j \hat{k}_{\lambda},\hat{k}_{\lambda}\rangle + \langle T_jT_j^* \hat{k}_{\lambda},\hat{k}_{\lambda}\rangle\right)^{\frac{1}{2}}\right)^2\\
               &\le &  \frac{1}{4}\left(\sum_{i=1}^n\big\langle (S_j^*S_j +S_jS_j^*)\hat{k}_{\lambda},\hat{k}_{\lambda}\big\rangle\right)\left(\sum_{i=1}^n\big\langle( T_j^*T_j+ T_jT_j^*)\hat{k}_{\lambda},\hat{k}_{\lambda}\big\rangle \right)\\
                 &= &\left(\sum_{i=1}^n\big\langle \big(\Re^2(S_j)+\Im^2(S_j)\big)\hat{k}_{\lambda},\hat{k}_{\lambda}\big\rangle\right)\left(\sum_{i=1}^n\big\langle\big(\Re^2(T_j)+\Im^2(T_j)\big)\hat{k}_{\lambda},\hat{k}_{\lambda}\big\rangle\right)\\ 
&\le&\left(\sum_{i=1}^n\textbf{ber}\big(\Re^2(S_j)+\Im^2(S_j)\big)\right)\left(\sum_{i=1}^n\textbf{ber}\big(\Re^2(T_j)+\Im^2(T_j)\big)\right)\\ &\le& \left(\sum_{i=1}^n\Big(\textbf{ber}^2(\Re(S_j))+\textbf{ber}^2(\Im(S_j))\Big)\right)\left(\sum_{i=1}^n\Big(\textbf{ber}^2(\Re(T_j))+\textbf{ber}^2(\Im(T_j))\Big)\right)\\&\le &(1+\sin^2\theta)^2\left(\sum_{i=1}^n\textbf{ber}^2(S_j)\right)\left(\sum_{i=1}^n\textbf{ber}^2(T_j)\right)\,\,\,\mbox{\big(by Lemma \ref{main}\big)}. 
    \end{eqnarray*} 
       Now, by taking the supremum over all $\lambda\in \mathscr{X}$, we obtain the desired inequality.
By taking $n=1$
 in Theorem \ref{commutating}, we derive  the following corollary.\end{proof}
\begin{cor}  Let $S,T\in \Pi_{\theta}^{\textbf{Ber},P}$ be such that $ST=TS$. Then
  \begin{eqnarray*}
         \textbf{ber}(ST)\le (1+\sin^2\theta)\textbf{ber}(S)\textbf{ber}(T).
  \end{eqnarray*}
\end{cor}

\begin{remark}
    Let $S,T\in \Pi_{\theta}^{\textbf{Ber},P}$  satisfy $ST=TS$. It follows that the upper bound of $\textbf{ber}(ST)$ varies with the parameter $\theta$, and interpolates  between $\textbf{ber}(S)\textbf{ber}(T)$ and  $2\textbf{ber}(S)\textbf{ber}(T)$.  When both $S$ and $T$ are positive, we have $\textbf{ber}(ST)\le \textbf{ber}(S)\textbf{ber}(T)$ and if $\theta\to \frac{\pi^{-}}{2}$,  this bound becomes $\textbf{ber}(ST)\le 2\textbf{ber}(S)\textbf{ber}(T)$.
\end{remark}
We conclude this section with the following inequality, which follows immediately from the previous result via mathematical induction. In a weak sense, this inequality can be viewed as a power inequality for a certain class of sectorial operators.

\begin{proposition}\label{PF1}
    If $T\in \Pi_{\theta}^{\textbf{Ber},P}$  then for all $n \in \mathbb N,$
  \begin{eqnarray*}
         \textbf{ber}(T^n)\le (1+\sin^2\theta)^{n-1}\textbf{ber}^n(T).
  \end{eqnarray*}
\end{proposition}

\section{Additional observations }\label{S6}
Let $\mathbb{D}=\left\{ z\in \mathbb{C}:|z|<1\right\}$ and $H(\mathbb{D})$ be the space of all holomorphic functions on $\mathbb{D}$. Suppose that $H^{\infty}(\mathbb{D})$ is the space of all bounded holomorphic functions on $\mathbb{D}$. Let $f\in H(\mathbb{D})$, then the Dirichlet integral of $f$ is defined by \[\mathcal{D}(f)=\frac{1}{\pi}\int_{\mathbb{D}}|f^{\prime}(z)|^2dA(z).\]
Let $\mathcal{D}$ denote the Dirichlet space consisting of all functions $f\in H(\mathbb{D})$ such that $\mathcal{D}(f)<\infty$.
For $f,g\in\mathcal{D}$, the inner product on $\mathcal{D}$ can be defined as 
\[\langle f,g\rangle=\langle f,g\rangle_{H^2(\mathbb{D})}+\frac{1}{\pi}\int_{\mathbb{D}} f^{\prime}(z)\overline{g^{\prime}(z)}dA(z),\] where $\langle \cdot,\cdot\rangle_{H^2(\mathbb{D})}$ is the usual inner product  on Hardy-Hilbert space $H^2(\mathbb D)$. The corresponding norm is defined by 
\[\|f\|^2=\|f\|_{H^2(\mathbb{D})}^2+\frac{1}{\pi}\int_{\mathbb{D}}|f^{\prime}(z)|^2dA(z).\]
The reproducing kernel for the Dirichlet space is $k_{w}(z)=\frac{1}{z\overline{w}}\log\left (\frac{1}{1-{z\overline{w}}}\right)$, where $z\in \mathbb{D}$, $w\in \mathcal{D}\setminus\{0\}$ and $k_0=1$. Also, the normalized reproducing kernel is $\|k_w\|^2=\frac{1}{|w|^2}\log\left(\frac{1}{1-|w|^2}\right)$. For more information on the Dirichlet space, we refer to \cite{Primer}.\\
The convexity of the Berezin range and some geometric properties of the finite rank operators on Hardy-Hilbert space and Bergman space have been discussed in \cite{Augustine}. In our first result, we prove the convexity of the Berezin range of the finite rank operator of the form $T(f)=\sum_{j=1}^n\langle f,g_j\rangle g_j,$ where $f,g_j \in \mathcal{D}.$
\begin{theorem}
    Let $T(f)=\sum_{j=1}^n\langle f,g_j\rangle g_j$, where $f,g_j\in \mathcal{D}$ be a finite rank operator on $\mathcal{D}$, with $g_j(z)=\sum_{m=1}^{\infty} a_{j,m}z^m$. Then the Berezin range of $T$ is convex in $\mathbb{C}$.
\end{theorem}
\begin{proof}
    Let $\lambda\in \mathbb{D}$. Then 
    \begin{eqnarray*}
Tk_{\lambda}=\sum_{j=1}^{n} \langle k_{\lambda},g_j\rangle g_j
=\sum_{j=1}^{n}\overline{g_j(\lambda)}g_j.  
    \end{eqnarray*}
    The Berezin transform of $T$ at $\lambda\in \mathbb{D}\setminus\{0\}$ is
      \begin{eqnarray*}
        \widetilde{T}(\lambda)=\langle T \hat{k}_{\lambda},\hat{k}_{\lambda}\rangle
        &=& |\lambda|^2\frac{1}{\log(\frac{1}{1-|\lambda|^2})}\langle Tk_{\lambda},k_{\lambda}\rangle\\
         &=& |\lambda|^2\frac{1}{\log(\frac{1}{1-|\lambda|^2})}\left\langle \sum_{j=1}^{n}\overline{g_j(\lambda)}g_j,k_{\lambda}\right\rangle\\
          &=& |\lambda|^2\frac{1}{\log(\frac{1}{1-|\lambda|^2})}\sum_{j=1}^{n}|g_j(\lambda)|^2.
    \end{eqnarray*} 
    For $\lambda=0$, $Tk_{\lambda}=\sum_{j=1}^n\overline{g_j(0)}g_j$. Therefore $\widetilde{T}(\lambda)=0$. 
    As $g_j\in\mathcal{D}$, so $g_j$ is a holomorphic function. Therefore $\widetilde{T}(\lambda)$ is a real continuous function on $\mathbb{D}$. Since $\mathbb{D}$ is a connected set, then the image of $\mathbb{D}$ under $\widetilde{T}$ is a connected set in $\mathbb{R}$. In $\mathbb{R}$, the connected sets are exactly the intervals (including $\mathbb R$ itself) and singleton sets.
    Hence, $\textbf{Ber}(T)$ is convex in $\mathbb{C}$.
\end{proof}
Now, by using the Berezin transformation formula, we establish a geometric characterization of the Berezin range for the finite-rank operator on $\mathcal{D}$.
\begin{theorem}
    Let $T(f)=\sum_{j=1}^{n}\langle f,g_j\rangle h_j$, where $f,g_j,h_j\in\mathcal{D}$ be a finite rank operator on $\mathcal{D}$ with $g_j(z)=\sum_{m=1}^{\infty}a_{j,m}z^{m}$ and $h_j(z)=\sum_{m=1}^{\infty}b_{j,m}z^{m}$. If the coefficients of $g_j$ and $h_j$ are real, then $\textbf{Ber}(T)$ is symmetric about the real axis. 
\end{theorem}
\begin{proof}
     Let $Tf=\sum_{j=1}^{n}\langle f,g_j\rangle h_j$. Then \[Tk_{\lambda}=\sum_{j=1}^n\langle k_{\lambda}, g_j\rangle h_j= \sum_{j=1}^n\overline{g_j(\lambda)}h_j.\]
    The Berezin transform of $T$ at $\lambda\in \mathbb{D}\setminus\{0\}$ is 
       \begin{eqnarray*}
        \widetilde{T}(\lambda)=\langle T \hat{k}_{\lambda},\hat{k}_{\lambda}\rangle
        &=& |\lambda|^2\frac{1}{\log(\frac{1}{1-|\lambda|^2})}\langle Tk_{\lambda},k_{\lambda}\rangle\\
         &=& |\lambda|^2\frac{1}{\log(\frac{1}{1-|\lambda|^2})}\left\langle \sum_{j=1}^{n}\overline{g_j(\lambda)}h_j,k_{\lambda}\right\rangle\\
          &=& |\lambda|^2\frac{1}{\log(\frac{1}{1-|\lambda|^2})}\sum_{j=1}^{n}\overline{g_j(\lambda)}h_j(\lambda).
    \end{eqnarray*} 
     For $\lambda=0$, $\widetilde{T}(\lambda)=0$. Putting $\lambda=re^{i\theta}$, so
     \begin{eqnarray*}
         \widetilde{T}(re^{i\theta})=r^2\frac{1}{\log\left(\frac{1}{1-r^2}\right)}\sum_{j=1}^n\left(\overline{\left(\sum_{m=1}^{\infty} a_{j,m}(re^{i\theta})^m\right)}\left(\sum_{m=1}^{\infty} b_{j,m}(re^{i\theta})^m\right)\right)
     \end{eqnarray*}
     Now, we compute $\overline{\widetilde{T}(re^{-i\theta}})$. Since $a_{j,m}$ and $b_{j,m}\in\mathbb{R}$, we have 
     \begin{eqnarray*}
         \overline{\widetilde{T}(re^{-i\theta}}) &=& \overline{r^2\frac{1}{\log\left(\frac{1}{1-r^2}\right)}\sum_{j=1}^n\left(\overline{\left(\sum_{m=1}^{\infty} a_{j,m}(re^{-i\theta})^m\right)}\left(\sum_{m=1}^{\infty} b_{j,m}(re^{-i\theta})^m\right)\right)}\\
         &=& r^2\frac{1}{\log\left(\frac{1}{1-r^2}\right)}\sum_{j=1}^n\left(\left(\sum_{m=1}^{\infty} a_{j,m}(re^{-i\theta})^m\right)\overline{\left(\sum_{m=1}^{\infty} b_{j,m}(re^{-i\theta})^m\right)}\right)\\
            &=& r^2\frac{1}{\log\left(\frac{1}{1-r^2}\right)}\sum_{j=1}^n\left(\overline{\left(\sum_{m=1}^{\infty} a_{j,m}(re^{i\theta})^m\right)}\left(\sum_{m=1}^{\infty} b_{j,m}(re^{i\theta})^m\right)\right)\\
            &=& \widetilde{T}(re^{i\theta}).
     \end{eqnarray*}
   Therefore,  $\textbf{Ber}(T)$ is closed under complex conjugation and so symmetric about the real axis. 
\end{proof}

The Berezin range of some weighted shift operators on the Hardy space and Bergman space is studied in \cite{Bulletin des}. Here, we are concerned about the Berezin range of weighted shift operators on the Dirichlet space.\\
Let $f(z)=\sum^{\infty}_{n=0}a_nz^n$ be an element of the Dirichlet space $\mathcal{D}$. The weighted shift operator on $\mathcal{D}$ is defined as $$T\left(\sum_{n=0}^{\infty}a_nz^n\right)=\sum_{n=0}^{\infty}a_n\beta_{n+1}z^{n+1},$$
where the weights $\{\beta_n\}$ is a bounded sequence in $\mathbb{C}.$ Since $\{\beta_n\}$ is bounded, so there exists $K\in\mathbb{R}$ such that $|\beta_n| \leq K,~~ n\in \mathbb N.$ Therefore, we obtain
  $$\sum_{n=0}^{\infty}(n+1)|a_n\beta_{n+1}|^2\leq K\sum_{n=0}^{\infty}(n+1)|a_n|^2<\infty.$$ 
  Thus $T$ is a bounded linear operator on $\mathcal{D}.$ In the next result, we consider some particular sequences and study properties such as the symmetry of the Berezin range of 
$T$ with respect to the real and imaginary axes. Moreover, we identify a particular weight for which the Berezin range is a disk centered at the origin with a specified radius.
\begin{theorem}
    Let $T$ be the weighted shift operator on $\mathcal{D}$. \\
   (i) If $\{\beta_n\}$ is a real bounded sequence, then $\textbf{Ber}(T)$ is symmetric about the real axis. \\
   (ii) If $\{\beta_n\}$ is a purely imaginary bounded sequence, then $\textbf{Ber}(T)$ is symmetric about the imaginary axis. \\
   (iii)  If $\{\beta_n=\frac{c}{n}\}$ with $c\in\mathbb{D}$ then the Berezin range of $T$ is a disk centred at origin and radius $|c|.$
\end{theorem}
\begin{proof}
    The Berezin transform of $T$ at $\lambda\in\mathbb{D}\setminus\{0\}$ is 
  \begin{eqnarray}
      \widetilde{T}(\lambda)&=&\langle T\hat{k_\lambda},\hat{k_\lambda}\rangle\nonumber\\&=&|\lambda|^2\frac{1}{\log\left(\frac{1}{1-|\lambda|^2}\right)}\langle Tk_\lambda,k_\lambda\rangle\nonumber\\&=&|\lambda|^2\frac{1}{\log\left(\frac{1}{1-|\lambda|^2}\right)}\left\langle\sum_{n=1}^{\infty}\frac{\overline{\lambda}^n}{\sqrt{n+1}}\beta_{n+1}z^{n+1},\sum_{n=1}^{\infty}\frac{\overline{\lambda}^n}{\sqrt{n+1}}z^n\right\rangle\nonumber\\&=&|\lambda|^2\frac{1}{\log\left(\frac{1}{1-|\lambda|^2}\right)}\lambda\left(\sum_{n=1}^{\infty}|\lambda|^{2n}\beta_{n+1}\right)\label{ttrree}.     
  \end{eqnarray}
  For $\lambda=0,$  $\widetilde{T}(0)=0.$\\
    $(i)$ Suppose that $\{\beta_n\}$ is a real bounded sequence. Then from the equality (\ref{ttrree}), we obtain that
    \begin{align*}
        \overline{\widetilde{T}(\lambda)}&=|\lambda|^2\frac{1}{\log\left(\frac{1}{1-|\lambda|^2}\right)}\overline{\lambda}\left(\sum_{n=1}^{\infty}|\lambda|^{2n}\beta_{n+1}\right)\\&=\widetilde{T}(\overline{\lambda}).
    \end{align*}
    Hence, $\textbf{Ber}(T)$ is symmetric about the real axis.\\
    $(ii)$ Suppose that $\{\beta_n\}$ is a purely complex bounded sequence, then from the equality \eqref{ttrree}, we get 
    \begin{align*}
          \overline{\widetilde{T}(\lambda)}&=-|\lambda|^2\frac{1}{\log\left(\frac{1}{1-|\lambda|^2}\right)}\overline{\lambda}\left(\sum_{n=1}^{\infty}|\lambda|^{2n}\beta_{n+1}\right)\\&=-\widetilde{T}(\overline{\lambda}).
    \end{align*}
    This implies that $\textbf{Ber}(T)$ is symmetric about the imaginary axis. \\
    $(iii)$ Suppose that $\beta_n=\frac{c}{n}$ for all $n \in \mathbb N,$ then from the relation (\ref{ttrree}), we get
    \begin{align*}
        \widetilde{T}(\lambda)&=|\lambda|^2\frac{1}{\log\left(\frac{1}{1-|\lambda|^2}\right)}\lambda\left(\sum_{n=1}^{\infty}|\lambda|^{2n}\frac{c}{n+1}\right)\\&=c\lambda, ~~~~\text{for $\lambda\in\mathbb{D}\setminus\{0\}$.}
    \end{align*}
    Again we have $\widetilde{T}(0)=0.$\\
    Therefore, $\textbf{Ber}(T)=\{z\in\mathbb{C}:|z|<|c|\},$ as desired.\\
\end{proof}

The results of this section indicate that the properties of the Berezin range of these operators on the Dirichlet space are largely similar to those observed in the Hardy-Hilbert space. For the case of $H^2(\mathbb D),$ as discussed in Section \ref{S3} it is possible to construct Berezin sectorial operators and(or) sectorial operators by applying certain shifts to composition-differentiation operators.  This naturally raises the question of whether similar constructions can be carried out for composition-differentiation operators on the Dirichlet space.
In this regard, \cite{AP} studied composition–differentiation operators on the Dirichlet space and determined the norms of operators induced by monomial symbols.

An important observation is that if $\widetilde{D_{\phi}}(re^{i\theta})=f(r)e^{i\theta}$ then the real-analytic property of the Berezin transform implies that whenever $\lim_{r \to 1^-}f(r)$ exists, the Berezin range of $D_{\phi}$ is a closed disk centered at the origin with radius 
$$r_0=\max_{r \in [0,1]}f(r).$$ In particular, if $r_0$ is strictly smaller than the numerical radius of $D_{\phi}$, this construction can be realized.

\section{Conclusion}

This paper introduces Berezin sectorial operators and demonstrates that they form a strictly broader class than classical sectorial operators. Our examples and inequalities highlight the effectiveness of the Berezin sectorial framework in refining Berezin number estimates and understanding the geometry of the Berezin range. These results highlight that the inequalities obtained for Berezin sectorial operators cannot be achieved by standard sectorial operators alone. While the Berezin range of several operators on the Dirichlet space exhibits properties analogous to those in the Hardy-Hilbert space, it remains an open problem whether composition-differentiation operators can be systematically constructed to be Berezin sectorial but not sectorial or both Berezin sectorial and sectorial but with different indices. Addressing this question may lead to further insights into the interaction between Berezin range geometry and Berezin number inequalities.


\section*{Declarations}	
\textit{Acknowledgements.} This work originated during the research visit of Dr. Anirban Sen to the Department of Mathematics at Jadavpur University, Kolkata, India. He gratefully acknowledges support from the Czech Science Foundation (GA CR, Grant No. 25-18042S).
Mr. Saikat Mahapatra thanks the UGC, Govt.\ of India, for financial support in the form of Fellowship. Miss Sweta Mukherjee is supported by the State Government Departmental Fellowship, Govt.\ of West Bengal.\\
\textit{Author Contributions:} All the authors contributed equally to this manuscript and reviewed it. \\
 \textit{Data Availability :} No datasets were generated or analysed during the current study. \\
\textit{Conflict of interest:} There is no competing interest.\\

\bibliographystyle{amsplain}

\end{document}